\documentclass[12pt,twoside]{article}

\usepackage[centertags]{amsmath}
\usepackage{amssymb}
\usepackage{amsthm}
\usepackage[misc]{ifsym}
\usepackage{mathrsfs}
\usepackage{kbordermatrix}
\usepackage{graphicx, subfigure, wrapfig}
\usepackage{tikz,multicol}
\usepackage{esvect}

\usepackage{pgfplots}
\pgfplotsset{compat=1.15}
\usepackage{mathrsfs}
\usetikzlibrary{arrows}

\date{August 5, 2021}

\begin{document}

\centerline{\Large{\bf Sharp bounds of the $A_\alpha$-spectral radii of mixed trees}}

\centerline{}

\centerline{\bf {Yen-Jen Cheng}}

\centerline{}

\centerline{Department of Applied Mathematics,}

\centerline{National Yang Ming Chiao Tung University,}

\centerline{1001 Ta Hsueh Road, Hsinchu, Taiwan.}

\centerline{Email addresses: yjc7755@nycu.edu.tw}

\centerline{}

\centerline{\bf {Louis Kao} }

\centerline{(Corresponding author)}

\centerline{}

\centerline{Department of Applied Mathematics,}

\centerline{National Yang Ming Chiao Tung University,}

\centerline{1001 Ta Hsueh Road, Hsinchu, Taiwan.}

\centerline{Email addresses: chihpengkao.am03@g2.nctu.edu.tw}

\centerline{}

\centerline{\bf {Chih-Wen Weng}}

\centerline{}

\centerline{Department of Applied Mathematics,}

\centerline{ National Yang Ming Chiao Tung University,}

\centerline{ 1001 Ta Hsueh Road, Hsinchu, Taiwan. }

\centerline{Email addresses: weng@math.nctu.edu.tw}

\newtheorem{Theorem}{Theorem}[section]

\newtheorem{Corollary}[Theorem]{Corollary}

\newtheorem{Lemma}[Theorem]{Lemma}

\newtheorem{Proposition}[Theorem]{Proposition}

\theoremstyle{definition}
\newtheorem{Example}[Theorem]{Example}
\newtheorem{Definition}[Theorem]{Definition}

\centerline{}

\begin{abstract} A mixed tree is a tree in which both directed arcs and undirected edges may exist. Let $T$ be a mixed tree with $n$ vertices and $m$ arcs, where an undirected edge is counted twice as arcs. Let $A$ be the adjacency matrix of $T$. For $\alpha\in[0,1]$, the matrix $A_\alpha$ of $T$ is defined to be $\alpha D^++(1-\alpha)A$, where $D^+$ is the  the diagonal out-degree matrix of $T$. The $A_\alpha$-spectral radius of $T$ is the largest real eigenvalue of $A_\alpha$. We will give a sharp upper bound and a sharp lower bound of the $A_\alpha$-spectral radius of $T$.
\end{abstract}



\section{Introduction}\label{s1}

A {\it mixed graph} is a graph in which directed and undirected edges between two distinct vertices may exist at most once.  To distinguish, we refer an {\it arc} to as  a directed edge in a mixed graph. An {\it undirected graph} is a mixed graph without arcs.  For a mixed graph $G$ of order $n$, let the vertex set of $G$ be $V(G)=[n]=\{1,2,\ldots,n\}$, and the set $E(G)$ collects all arcs $\overrightarrow{ij}$ and all undirected  edges $ij$ in $G$, where $i, j\in V(G)$ are distinct. The {\it size} of a mixed graph $G$ is defined to be the number of arcs plus twice the number of undirected edges in $G$. The {\it adjacency matrix} $A=(a_{ij})$ of $G$ is a square $01$-matrix of order $n$ defined by $a_{ij}=1$ if and only if $\vv{ij}$ or $ij$ in $E(G)$. The diagonal matrix $D^+={\rm diag}(d^+_1, d^+_2, \ldots, d^+_n)$  is called the {\it out-degree matrix}  of $T$, where $d^+_i=|\{j~:~\vv{ij}\in E(G){~\rm or~}ij\in E(G)\}|$. For real number $\alpha\in[0,1]$, the matrix $A_\alpha(G)$ of $G$ is defined to be $\alpha D^++(1-\alpha)A$. The concepts of $A_\alpha$ matrix of graphs were first introduced by Nikiforov\cite{n:17} in 2017 and then liu et al.\cite{l:19} start to consider the $A_\alpha$ matrix for digraphs. Since $A_\alpha$ is nonnegative and it is well known that a nonnegative matrix has a real eigenvalue,  let $\rho_\alpha(G)$ denote
the largest real eigenvalue $\rho(A_\alpha(G))$  of the $A_\alpha$  matrix $A_\alpha(G)$ of $G$, and refer  $\rho_\alpha(G)$ to as the {\it $A_\alpha$-spectral radius}, or $\alpha$-{\it index} of $G$. For the previous studies on $A_\alpha$-spectral radii of undirected graphs and mixed graphs, see \cite{gb:20,gz:20,lcm:19,nprs:17,r:20,xsw:20}.

The {\it underlying graph} of a mixed graph $G$ is the undirected graph obtained from $G$ by removing the directions on arcs.
The {\it distance} $\partial(a,b)$ for vertices $a,b$ in $G$ is their distance in the underlying graph of $G$. The {\it diameter} of $G$ is defined to be $\max_{a,b\in V(G)}\partial(a,b)$.
Similarly the  {\it mixed tree}, {\it mixed path}, {\it mixed star} are defined as mixed graphs whose underlying graphs are
tree, path, and star, respectively, where a star is a tree of diameter at most $2$. We denote the mixed path of order $k$ and size $2k-2$ by $P_k$.
 The main goal of this paper is to find the sharp upper bound and sharp lower bound of the $A_\alpha$-spectral radii of mixed trees of order $n$ and size $m$, where $n-1\leq m\leq 2n-2$. The following two theorems are our main results.

\begin{Theorem}\label{upr}
If $\alpha\in [0, 1]$ and $T$ is a mixed tree of order $n$ and size $m$, then $$\rho_\alpha(T)\leq\frac{1}{2}\left(\alpha n+\sqrt{\alpha^2n^2-4\alpha^2(n-1)+4(1-\alpha)^2(m-n+1)}\right).$$ Moreover, every mixed star of order $n$ and size $m$ with maximum out-degree $n-1$ attains the upper bound.
\end{Theorem}

\begin{Theorem}\label{lwr}
If $T$ is a mixed tree of order $n$ and size $m$, and set $k=\lceil \frac{n}{2n-m-1}\rceil$ then $$\rho_\alpha(T)\geq \rho_\alpha(P_k).$$
Moreover, the lower bound is attained when $T=P_n$.
\end{Theorem}

It worths mentioning that in the special case $m=2n-2$, Theorem~\ref{upr} and Theorem~\ref{lwr} are proved in \cite{nprs:17}.
The main tool in the proof of Theorem~\ref{upr} is the Kelmans transformation for matrices \cite{kw:21} which will be introduced in Section~\ref{s2}.
A partially ordered set (poset) $\mathcal{G}(n, m)$ of mixed graphs of order $n$ and size $m$ that respects the order of $A_\alpha$-spectral radii is introduced
and the maximal elements in $\mathcal{G}(n, m)$ are characterized in Section~\ref{s2.5}.
The maximal elements in the subposet $\mathcal{T}(n, m)$ of  $\mathcal{G}(n, m)$  induced on the mixed trees are determined in Section~\ref{s2.7}.
Theorem~\ref{upr}  will be proven in Section~\ref{s3}.  Theorem~\ref{lwr} is essentially a consequence of \cite{nprs:17}. We mention this in Section~\ref{s4}.

\section{Preliminaries}\label{s2}

To compare the $A_\alpha$-spectral radii of mixed graphs, we first introduced a useful tool called the Kelmans transformation. The Kelmans transformation,
of an undirected graph, or called graph compression, was defined by A.K. Kelmans in 1981\cite{k:81}. The authors Kao and Weng have generalized it into a matrix version in \cite{kw:21}, and here we will further generalize the transformation to fit the $A_\alpha$ matrix.

Let $C=(c_{ij})$ be a nonnegative square matrix of order $n$.
Fix a $2$-subset $\{a, b\}$ of $[n]$, and assume that $c_{ab}=c_{ba}$.
Choose $k$ such that
\begin{equation}\label{e1}
\max(0,c_{bb}-c_{aa})\leq k\leq c_{bb}\end{equation}
and for $i\in [n]-\{a, b\}$, choose  $t_i$ and $s_i$ such that
\begin{equation}\label{e2}
\max(0, c_{ib}-c_{ia})\leq t_i\leq c_{ib}, \quad \max(0, c_{bi}-c_{ai})\leq s_i\leq c_{bi}.
\end{equation}
We define a new matrix $C_b^a=C_b^a(t_i; s_i; k)$ of order $n$ from $C$ by shifting the portion $k$ from $c_{bb}$ to $c_{aa}$, the portion $t_i$ of $c_{ib}$ to $c_{ia}$
and the portion $s_i$ of $c_{bi}$ to $c_{ai}$ such that in the new matrix $C_b^a=(c'_{ij})$  we have $c'_{aa}\geq c'_{bb}$, $c'_{ia}\geq c'_{ib}$, and  $c'_{ai}\geq c'_{bi}$, for all $i\in [n]-\{a, b\}$. The following is an illustration of $C_b^a$:
$$C_b^a=~~ \kbordermatrix{~        & & j & a &   b\\
                           &~~&              &   &     \\
                       i    &&      c_{ij}   & c_{ia}+t_i  &  c_{ib}-t_i \\
                           &&          &             &       \\
                       a   &&   c_{aj}+s_j          & c_{aa}+k  & c_{ab}   \\
                       b   &&   c_{bj}-s_j          &c_{ba}   & c_{bb}-k    \\ }\qquad
\left\{
  \begin{array}{l}
    i, j\in [n]-\{a, b\},  \\
    c_{ab}=c_{ba}, \\
    \max(0, c_{ib}-c_{ia})\leq t_i\leq c_{ib}, \\
    \max(0, c_{bj}-c_{aj})\leq s_j\leq c_{bj}, \\
   \max(0,c_{bb}-c_{aa})\leq k\leq c_{bb}.
  \end{array}
\right.$$

\noindent Formally, the matrix $C_b^a=(c'_{ij})$ is defined from $C=(c_{ij})$ by setting
\begin{equation}\label{C'}c'_{ij}=\left\{
          \begin{array}{ll}
            c_{ij}, & \hbox{if $i, j\in [n]-\{a, b\}$ or $(i, j)\in \{(a, b), (b, a)\}$;} \\
            c_{ia}+t_i, & \hbox{if $j=a$ and $i\in [n]-\{a, b\}$;} \\
            c_{ib}-t_i, & \hbox{if $j=b$ and $i\in [n]-\{a, b\}$;} \\
            c_{aj}+s_j, & \hbox{if $i=a$ and $j\in [n]-\{a, b\}$;} \\
            c_{bj}-s_j, & \hbox{if $i=b$ and $j\in [n]-\{a, b\}$;} \\
            c_{aa}+k,  & \hbox{if $i=j=a$;} \\
            c_{bb}-k,  & \hbox{if $i=j=b$.} \\
          \end{array}
        \right.
\end{equation}
The matrix $C_b^a$ is referred to as the  {\it Kelmans transformation of $C$ from $b$ to $a$ with respect to $(t_i;s_i;k)$}.
In the above setting, if $C=(c_{ij})$ is the adjacency matrix of an undirected graph of order $n$ and assume that $C_b^a$ is also a symmetric binary matrix, then by the assumptions in (\ref{e1})-(\ref{e2}), $t_i=\max(0, c_{ib}-c_{ia})=\max(0, c_{bi}-c_{ai})=s_i\in \{0, 1\}$ and $k=0$ are uniquely determined from $C$. In this situation, we don't need to mention $(t_i;s_j;k)$ and the Kelmans transformation $C_b^a$ of $C$ from $b$ to $a$ is the adjacency matrix of the graph obtained by the Kelmans transformation of an undirected graph defined by A.K. Kelmans \cite{k:81}.

P. Csikv\'{a}ri proved that the largest real eigenvalues of adjacency matrices will not be decreased after a Kelmans transformation of an undirected graph \cite{c:09}. The following theorem is a generalization of this result to a nonnegative matrix which is not necessary to be symmetric.

\begin{Theorem}\label{thm1}\cite{kw:21} Let $C=(c_{ij})$ denote a nonnegative square matrix of order $n$ such that $c_{ab}=c_{ba}$ for some $1\leq a, b\leq n$. Choose $k, t_i, s_i$ for $i\in [n]-\{a, b\}$ that satisfy (\ref{e1}),(\ref{e2}). Let $C_b^a=C_b^a(t_i; s_i; k)$ be the Kelmans transformation from $b$ to $a$ with respect to $(t_i;s_j;k)$. Then $\rho(C)\leq \rho(C_b^a).$ \qed
\end{Theorem}

It worths mentioning that Theorem~\ref{thm1} appearing in \cite{kw:21}  has the additional assumption $c_{aa}=c_{bb}$.
The interested reader might trace the proof in \cite{kw:21} and find that the same proof works fine for this slightly more general situation here.

As in the case of undirected graph, if $C$ in Theorem~\ref{thm1} is the adjacency matrix of a mixed graph $G$ and assume that $C_b^a$ is also an adjacency matrix of some mixed graph, then $t_i, s_i\in \{0, 1\}$ and $k=0$ are uniquely determined from $C$.
We use $G_b^a$ to denote the mixed graph whose adjacency matrix is $C_b^a$ and called $G_b^a$ the {\it Kelmans transformation} of mixed graph $G$ from $b$ to $a$. Notice that when the notation $G_b^a$ appears, we always assume that $a, b\in V(G)$ are distinct and have no arc, i.e. $\vv{ab}\notin E(G)$ and $\vv{ba}\notin E(G)$. Figure~\ref{fig1} shows how the Kelmans transformation on mixed graph works.

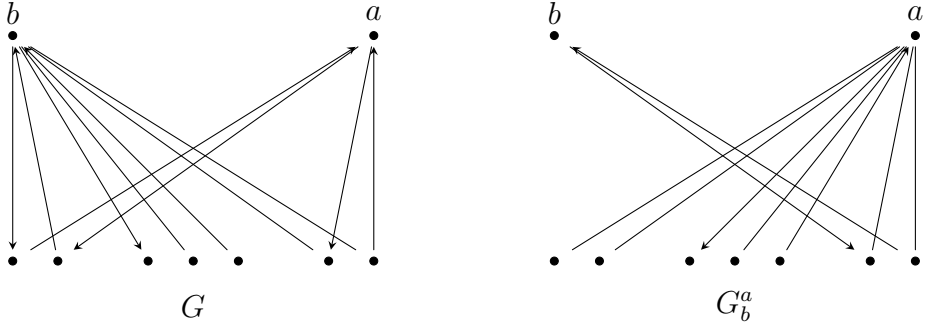
\begin{figure}
\centering
\begin{tikzpicture}[line cap=round,line join=round,>=triangle 45,x=0.6cm,y=0.6cm]

\draw  [stealth-](0.05,9.75)-- (0.95,5.25);
\draw  [-stealth](0.15,9.75)-- (2.85,5.25);
\draw  (0.2,9.75)-- (3.8,5.25);
\draw  [stealth-](0.25,9.75)-- (4.75,5.25);
\draw  (0.35,9.75)-- (6.65,5.25);
\draw  (7.6,5.25)-- (0.4,9.75);
\draw  [-stealth](0.4,5.25)-- (7.60,9.75);
\draw  [stealth-](1.35,5.25)-- (7.65,9.75);
\draw  [stealth-](7.05,5.25)-- (7.95,9.75);
\draw  [-stealth](8.,5.25)-- (8.,9.75);
\draw [-stealth] (0.,9.75) -- (0.,5.25);

\draw [fill=black] (0.,10.) circle (1.5pt);
\draw [fill=black] (0.,5.) circle (1.5pt);
\draw [fill=black] (1.,5.) circle (1.5pt);
\draw [fill=black] (3.,5.) circle (1.5pt);
\draw [fill=black] (4.,5.) circle (1.5pt);
\draw [fill=black] (5.,5.) circle (1.5pt);
\draw [fill=black] (7.,5.) circle (1.5pt);
\draw [fill=black] (8.,5.) circle (1.5pt);
\draw [fill=black] (8.,10.) circle (1.5pt);

\draw  [-stealth](19.75,9.75)-- (15.25,5.25);
\draw  (19.8,9.75)-- (16.2,5.25);
\draw  [stealth-](19.85,9.75)-- (17.15,5.25);
\draw  [-stealth](12.35,9.75)-- (18.65,5.25);
\draw  [-stealth](19.6,5.25)-- (12.4,9.75);
\draw  (12.4,5.25)-- (19.60,9.75);
\draw  (13.35,5.25)-- (19.65,9.75);
\draw  (19.05,5.25)-- (19.95,9.75);
\draw  (20.,5.25)-- (20.,9.75);

\draw [fill=black] (12.,10.) circle (1.5pt);
\draw [fill=black] (12.,5.) circle (1.5pt);
\draw [fill=black] (13.,5.) circle (1.5pt);
\draw [fill=black] (15.,5.) circle (1.5pt);
\draw [fill=black] (16.,5.) circle (1.5pt);
\draw [fill=black] (17.,5.) circle (1.5pt);
\draw [fill=black] (19.,5.) circle (1.5pt);
\draw [fill=black] (20.,5.) circle (1.5pt);
\draw [fill=black] (20.,10.) circle (1.5pt);

\draw (0,10.5) node{$b$};
\draw (8,10.5) node{$a$};
\draw (12,10.5) node{$b$};
\draw (20,10.5) node{$a$};
\draw (4,4) node{$G$};
\draw (16,4) node{$G_b^a$};

\end{tikzpicture}
\caption{Kelmans transformation on mixed graph $G$.}
\label{fig1}
\end{figure}

For a mixed graph $G$, let $N^+_{G}(u):=\{v \colon \vv{uv}\in E(G)~{\rm or~} uv\in E(G)\}$ be the set of {\it out-neighbors} of $u$, $N^-_{G}(u):=\{v \colon \vv{vu}\in E(G)~{\rm or~} uv\in E(G)\}$ be the set of {\it in-neighbors} of $u$, and  $N_G(u):=N^+_{G}(u)\cup N^-_{G}(u)$ be the set of
{\it neighbors} of $u$.   The number $d^+_G(u):=|N^+_{G}(u)|$ is called the {\it out-degree} of $u$ in $G$, and  the number $d_G(u):=|N^+_{G}(u)|+|N^-_{G}(u)|$ is called the {\it degree} of $u$ in $G$.  The sequence $d(G):=(d_G(u))_{u\in V(G)}$  in descending order is called the {\it degree sequence} of $G$. In the rest of this paper, the order of degree sequences are considered in dictionary order, that is, $(a_1,a_2,\ldots,a_n)>(b_1,b_2,\ldots,b_n)$ if  for the minimum $i$ with $a_i\not=b_i$, we have $a_i>b_i$. The (i) of the following lemma in mixed graph is generalized from its undirected graph version \cite{k:81}.

\begin{Lemma}\label{l2.2} Let $G$ be a mixed graph and distinct $a, b\in V(G)$ have no arc.  Then the following (i)-(ii) hold.
\begin{enumerate}
\item[(i)] The involution $f:V(G_b^a)\rightarrow V(G_a^b)$ defined by
$$f(x)=\left\{
       \begin{array}{ll}
         a, & \hbox{if $x=b$;} \\
         b, & \hbox{if $x=a$;} \\
         x, & \hbox{otherwise}
       \end{array}
     \right.$$
is a graph isomorphism from $G_b^a$ to $G_a^b$.
\item[(ii)]  In dictionary order, $d(G_b^a)\geq d(G)$.  Moreover, the following (a)-(c) are equivalent.
\begin{enumerate}
\item[(a)] $d(G_b^a)=d(G)$;
\item[(b)] $G$ is isomorphic to $G_b^a$;
\item[(c)] $N^+_G(a)-\{b\}\subseteq N^+_G(b)-\{a\}$ and $N^-_G(a)-\{b\}\subseteq N^-_G(b)-\{a\}$; or  $N^+_G(b)-\{a\}\subseteq N^+_G(a)-\{a\}$ and $N^-_G(b)-\{a\}\subseteq N^-_G(a)-\{b\}$.
\end{enumerate}
\end{enumerate}
\end{Lemma}

\begin{proof} Excluding the two vertices $a, b$ which are either with an undirected edges or without any directed arcs by the assumption, we have the following three observations of neighbor sets from the definition of Kelmans transformation on $G$ from $b$ to $a$.
(1) the set of out-neighbors (resp. in-neighbors) of $b$ in $G_b^a$ is the union of the set of  out-neighbors  (resp. in-neighbors) of $a$ in $G$ and the set of  out-neighbors  (resp. in-neighbors) of $b$ in $G$; (2) the set of  out-neighbors  (resp. in-neighbors) of $a$ in $G_b^a$ is  the intersection of set of the out-neighbors  (resp. in-neighbors) of $a$ in $G$ and the set of out-neighbors  (resp. in-neighbors) of $b$ in $G$;
(3) the set of out-neighbors (resp. in-neighbors) of  $x\not=a, b$ in $G_b^a$ is the same as that in $G$.
From the above three observations, we find that vertices $a, b, x$ in $G_b^a$ play the role of $b, a, x$ respectively in $G_a^b$. This proves (i).
\medskip

\noindent (ii) In the proof of (i), we also have $d_G(x)=d_{G_b^a}(x)$ for $x\in V(G)-\{a, b\}$
and in dictionary order $(d_{G_b^a}(a),  d_{G_b^a}(b))\geq (\max(d_G(a), d_G(b)), \min(d_G(a), d_G(b)))$, together implying $d(G_b^a)\geq d(G)$. Next we prove that (a), (b) and (c) are equivalent.
\medskip

\noindent ((b) $\Rightarrow$ (a)) This is clear.
\medskip

\noindent ((a) $\Rightarrow$ (c)) Suppose $d(G_b^a)=d(G)$. From the proof of (ii) above, we have
$\{d_G(a), d_G(b)\}=\{d_{G_b^a}(a), d_{G_b^a}(b)\}$. If $d_G(a)=d_{G_b^a}(b)$ then  $d_G(b)=d_{G_b^a}(a)\geq d_{G_b^a}(b)=d_G(a)$, which implies $N^+_G(a)-\{b\}\subseteq N^+_G(b)-\{a\}$ and $N^-_G(a)-\{b\}\subseteq N^-_G(b)-\{a\}$. If $d_G(a)=d_{G_b^a}(a)$ then  $d_G(b)=d_{G_b^a}(b)$, which implies $N^+_G(b)-\{a\}\subseteq N^+_G(a)-\{b\}$ and $N^-_G(b)-\{a\}\subseteq N^-_G(a)-\{b\}$.
\medskip

\noindent ((c) $\Rightarrow$ (b)) If  $N^+_G(a)-\{b\}\subseteq N^+_G(b)-\{a\}$ and $N^-_G(a)-\{b\}\subseteq N^-_G(b)-\{a\}$ then $G=G_a^b$ and the latter is isomorphic to $G_b^a$  by (i). If  $N^+_G(b)-\{a\}\subseteq N^+_G(a)-\{b\}$ and $N^-_G(b)-\{a\}\subseteq N^-_G(a)-\{b\}$ then $G=G_b^a$.
\end{proof}

For a  square matrix $M$, let ${\rm char}(M):={\rm det}(\lambda I-M)$ denote the {\it characteristic polynomial} of $M$. The following lemma
is immediate from the definition of  characteristic polynomial of $M$.

\begin{Lemma}\label{l2.6}
For an $n\times n$ nonnegative matrix $M$, if $$M=\begin{pmatrix}
M_1& M_2\\ 0&  M_3
\end{pmatrix}\quad {\rm ~or~}\quad \begin{pmatrix}
M_1& 0\\ M_2 &  M_3
\end{pmatrix}$$ where $M_1,M_3$ are square matrices, then ${\rm char}(M)={\rm char}(M_1)\cdot {\rm char}(M_2)$.\qed
\end{Lemma}

For an $n\times n$ matrix $M$ and a partition $\Pi=\{\pi_1,\pi_2\ldots,\pi_\ell\}$ of $[n]$, the $\ell\times\ell$ matrix $\Pi(M)=(m'_{ij})$, where $$m'_{ab}=\frac{1}{|\pi_a|}\sum_{i\in\pi_a,j\in\pi_b}m_{ij}\qquad (1\leq a, b\leq \ell),$$
is called the {\it quotient matrix} of $M$ with respect to $\Pi$. Furthermore, if for all $1\leq a,b\leq\ell$ and $i\in\pi_a$,
$\sum_{j\in\pi_b}m_{ij}=m'_{ab}$, then $\Pi(M)$ is called the {\it equitable quotient matrix} of $M$ with respect to $\Pi$. The following well-known lemma is useful on the calculating of spectral index. See \cite{ap:18,clw:22} for recent proofs.

\begin{Lemma}\label{equ_quo}
If $\Pi(M)$ be the equitable quotient matrix of a nonnegative matrix $M$, then $\rho(M)=\rho(\Pi(M))$. \qed
\end{Lemma}

The following is a well-known consequence of Perron–Frobenius theorem \cite{bh:12}.

\begin{Lemma}\label{perron}
If $N$ is a nonnegative square matrix and $M$ is a nonnegative matrix of the same size with $M\leq N$, or $M$ is a nonnegative submatrix of $N$  then $\rho(M)\leq \rho(N)$. \qed
\end{Lemma}

\section{The poset $\mathcal{G}(n, m)$ of mixed graphs}\label{s2.5}

For a mixed graph $G$ of order $n$ and size $m$, let $[G]$ denote the set of mixed graphs that are isomorphic to $G$.
Let \begin{equation}\label{iso}
\mathcal{G}(n, m):=\{ [G]~\colon~G~\hbox{is a mixed graph of order $n$ and size $m$}\}.
\end{equation}
We will define a reflexive and transitive relation $\leq$ in  $\mathcal{G}(n, m)$ as follows.
\begin{Definition} Let $\leq $ be the relation in  $\mathcal{G}(n, m)$ such that for all $[G], [H]\in \mathcal{G}(n, m)$,
$[G]\leq [H]$ if and only if  $H$ is isomorphic to $G$, or $H$ is isomorphic to a graph which is obtained from $G$ by a finite sequence of Kelmans   transformations.
\end{Definition}

\begin{Lemma}\label{l2.3} $(\mathcal{G}(n, m), \leq)$ is a partially ordered set (poset).
\end{Lemma}

\begin{proof} The relation $\leq$ is reflexive and transitive from its definition, so
we only need to prove the anti-symmetric property.
Suppose $[G]\leq [H]$ and $[H]\leq [G]$, where $[G], [H], \in G(n, m)$. Then $d(G)\leq d(H)\leq d(G)$ by Lemma~\ref{l2.2}(ii). Hence $d(G)=d(H)$. By Lemma~\ref{l2.2}(ii)(a)$\Rightarrow$(b), we have $[G]=[H]$.
\end{proof}

\begin{Lemma}\label{l2.4} Let $\alpha\in [0, 1]$, $[G]\in \mathcal{G}(n, m)$ with distinct vertices $a, b\in V(G)$ having no arc, adjacency matrix $A=(c_{ij})$ and $A_\alpha$ matrix $A_\alpha(G)$ of $G$. Set $k:=\alpha |N^+_G(b)-N^+_G(a)|$, $t_i=(1-\alpha) \max(0, c_{ib}-c_{ia})$ and $s_i=(1-\alpha) \max(0, c_{bi}-c_{ai})$ for $i\in V(G)-\{a, b\}$. Then the Kelmans transformation matrix $A_\alpha(G)_b^a$ of $A_\alpha(G)$ from $b$ to $a$ with respect to $(t_i;s_i;k)$ is the $A_\alpha$ matrix $A_\alpha(G_b^a)$ of $G_b^a$, i.e.,
$$A_\alpha(G)_b^a=A_\alpha(G_b^a).$$
\end{Lemma}
\begin{proof} We only need to check that the $ij$ entries in matrices $A_\alpha(G)_b^a$ and $A_\alpha(G_b^a)$   are equal for one of $i, j$ in $\{a, b\}$. Indeed they are equal from the setting listed in the order $aa$, $bb$, $ia$, $ib$, $aj$ and $bj$ below:
\begin{align*}
 \alpha d^+_G(a)+k=&\alpha d^+_{G_b^a}(a),\\
\alpha d^+_G(b)-k=&\alpha d^+_{G_b^a}(b),\\
 (1-\alpha)c_{ia}+t_i =&(1-\alpha)(c_{ia}+\max(0, c_{ib}-c_{ia})),\\
(1-\alpha)c_{ib}-t_i=&(1-\alpha)(c_{ib}-\max(0, c_{ib}-c_{ia})),\\
(1-\alpha)c_{ia}+s_j=&(1-\alpha)(c_{aj}+\max(0, c_{bj}-c_{aj}),\\
(1-\alpha)c_{ib}-s_j=&(1-\alpha)(c_{bj}-\max(0, c_{bj}-c_{aj}),
\end{align*}
where $i, j\in V(G)-\{a, b\}$.
\end{proof}

\begin{Proposition}\label{p2.5} If $\alpha\in [0, 1]$, and $[G], [H]\in G(n, m)$ such that $[G]\leq [H]$, then $\rho_\alpha(G)\leq \rho_\alpha(H).$
\end{Proposition}
\begin{proof}
We might assume $H=G_b^a$ by Lemma~\ref{l2.3}.
Applying Theorem~\ref{thm1} and Lemma~\ref{l2.4}, we have
$$\rho_\alpha(G)=\rho(A_\alpha(G))\leq \rho(A_\alpha(G)_b^a)=\rho(A_\alpha(G_b^a))=\rho_\alpha(H).$$
\end{proof}

\section{The poset $\mathcal{T}(n, m)$ of mixed trees}\label{s2.7}

Let $n, m\in \mathbb{N}$ with $n-1\leq m\leq 2n-2$, $$\mathcal{T}(n, m):=\{[T]\in \mathcal{G}(n, m)\colon T \hbox{~is a mixed tree}\}.$$
The set $\mathcal{T}(n, m)$ is not closed under Kelmans transformations. We need the following lemma.

\begin{Lemma}\label{tree} Let $[T]\in \mathcal{T}(n, m)$ with distinct $a, b\in V(T)$ having no arc.
Then $[T_b^a]\in \mathcal{T}(n, m)$ if and only if $ab\in E(G)$ or $\partial(a, b)=2$ and the unique vertex $x\in V(G)$ with $\partial(a,x)=\partial(x, b)=1$ satisfying one of (i) $ax\in E(G)$ is an undirected edge, (ii) $xb\in E(G)$ is an undirected edge, (iii) $\overrightarrow{ax}, \overrightarrow{bx}\in E(G)$ are arcs or (iv) $\overrightarrow{xa}, \overrightarrow{xb}\in E(G)$ are arcs.
\end{Lemma}
\begin{proof}
The assumption implies $\partial(a, b)\geq 1$ and if $\partial(a, b)=1$ then $ab\in E(G)$ is an undirected edge.
If $\partial(a, b)=2$ and the necessary condition about $x$ fails then $a,b$ belong to different components of the underlying graph of $T_b^a$, so $T_b^a$ is not a mixed tree.
If $\partial(a, b)\geq 3$ then the underlying graph of $T_b^a$ contains a cycle of order $\partial(a, b)$, so $T_b^a$ is not a mixed tree.

On the other hand, it is straightforward to observe that $[T_b^a]\in \mathcal{T}(n, m)$ when $a,b$ satisfy the conditions.
\end{proof}

We use the notation $a-b$, $a-x\rightarrow b$,  $a-x\leftarrow b$, $a\leftarrow x-b$, $a\rightarrow x-b$, $a\rightarrow x\leftarrow b$ and $a\leftarrow x \rightarrow b$ to denote the seven situations in the necessary condition of Lemma~\ref{tree}. We then give $\mathcal{T}(n, m)$ a poset structure by extending $[T]\leq [T_b^a]$ for any $[T]\in \mathcal{T}(n, m)$ and any $a, b\in V(T)$ that satisfy one of the seven situations.

\begin{Proposition}\label{p3.2}
Let $[T]\in\mathcal{T}(n,m).$ Then $[T]$ is a maximal element in $\mathcal{T}(n,m)$  if and only if $T$ is a mixed star or
$T$ is a mixed tree without undirected edges (i.e. $m=n-1$) and whenever the subgraph $a\rightarrow x \leftarrow b$ or $a\leftarrow x \rightarrow b$ appears in $T$, one of $a$ and $b$ is a leaf.
\end{Proposition}
\begin{proof}
$(\Leftarrow)$ If $T$ is a mixed star,  and one of $a-b$, $a-x\rightarrow b$,  $a-x\leftarrow b$, $a\leftarrow x-b$, $a\rightarrow x-b$, $a\rightarrow x\leftarrow b$ and $a\leftarrow x \rightarrow b$ appearing in $T$, then one of $a$ or $b$ is a leaf, so Lemma~\ref{l2.2}(iic) with $G=T$ holds, which implies that $T_b^a$ is isomorphic to $T$.
If $T$ is a mixed tree without undirected edges, then we only need to consider $a\rightarrow x\leftarrow b$ and $a\leftarrow x \rightarrow b$ in $T$.
By the assumption $a$ or $b$ is a leaf and by the same reason as above, $T_b^a$ is isomorphic to $T$. Hence in both cases, $[T]$ is a maximal element in $\mathcal{T}(n,m)$.
\medskip

\noindent $(\Rightarrow)$
Let $[T]$ be a maximal element in $\mathcal{T}(n, m)$ such that $T$ is not a mixed star, so $T$ has diameter at least $3$. Keeping in mind that the maximality of $[T]$ implies that  Lemma~\ref{l2.2}(iic) with  $G=T$ holds for $a, b\in V(T)$ satisfying the necessary conditions $a-b$, $a-x\rightarrow b$,  $a-x\leftarrow b$, $a\leftarrow x-b$, $a\rightarrow x-b$, $a\rightarrow x\leftarrow b$ or $a\leftarrow x \rightarrow b$ of Lemma~\ref{tree}, thus at least one of $a$ or $b$ is a leaf. To exclude the situations $a-b$, $a-x\rightarrow b$,  $a-x\leftarrow b$, $a\leftarrow x-b$ and $a\rightarrow x-b$, on the contrary,  suppose that $T$ contains an undirected edge $uv$ with leaf $u$. Since diameter of $T$ at least $3$,  we have another two vertices $y,z\in V(T)$ such that $\partial(v, y)=\partial(y, z)=1$ and $\partial(u, z)=3$. Since $v, y$ are not leafs in $T$,  they have an arc, say $v\rightarrow y$ (similar for $v\leftarrow y$) in  $E(T)$. Hence $T_y^u\in \mathcal{T}(n, m)$ is well-defined,  $v\in (N^+_T(v)-\{y\})-(N_T(y)-\{u\})$, and $z\in N_T(y)-N_T(u)$, a contradiction to the maximality of $[T]$. Thus $T$ has no undirected edges.
\end{proof}

\section{The upper bound of $\rho_\alpha(T)$}\label{s3}

If an arc in a mixed tree $T$ is deleted then we have two mixed trees. Thus
if the arcs in a mixed tree $T$ of order $n$ and size $m$ are all removed, then the remaining is an undirected graph without cycles with $2n-m-1$ components.
We call these $2n-m-1$ components the {\it components} of $T$.

\begin{Lemma}\label{char}
If $\alpha\in [0, 1]$ and $[T]\in \mathcal{T}(n, m)$ and $T$ has components $C_1$, $C_2$, $\ldots$, $C_t$, then
$${\rm char} (A_\alpha(T))=\prod_{i\in [t]} {\rm char} (A_\alpha(T)[C_i]),$$
where $A_\alpha(T)[C_i]$ is the principal submatrix of $A_\alpha(T)$ restricted to $C_i$.
 \end{Lemma}
\begin{proof}
If $\overrightarrow{ij}\in E(T)$ is deleted to obtained two mixed trees with vertex sets $V$ and $W$, then besides $\overrightarrow{ij}$ there is no
arcs or undirected edges between a vertex in $V$ and a vertex in $W$. With $M=A_\alpha(T)$, $M_1=M[V]$, $M_2=M[W]$, we find $M$ satisfies the assumption of Lemma~\ref{l2.6}.
Hence ${\rm char}(M)={\rm char}(M_1)\times {\rm char}(M_2)$. We have the lemma by use this process on $M_1$ and $M_2$, and repeating again until each matrix is corresponding to a component of $T$.
\end{proof}

Note that $A_\alpha(T)[C_i]$ in Lemma~\ref{char} is not the $A_\alpha$ matrix of the component $C_i$ in $T$.

\begin{Corollary}\label{n-1}
If $\alpha\in[0,1]$ and $[T]\in\mathcal{T}(n,n-1)$, then
$${\rm char} (A_\alpha(T))=\prod_{i\in [n]}(\lambda-\alpha d_i^+).$$
\end{Corollary}
\begin{proof}
For $[T]\in\mathcal{T}(n,n-1)$, each vertex forms a component. Since $A_\alpha(T)[\{i\}]$ is an $1\times 1$ matrix with entries $\alpha d_i^+$, the result is straightforward from Lemma~\ref{char}.
\end{proof}

\begin{Proposition}\label{p3.3} Let $S$ be mixed star of order $n$, size $m$ and maximum out-degree $m-n+k+1$ for some $0\leq k\leq 2n-m-2$. Then for $\alpha\in [0, 1]$, the $A_\alpha$-spectral radius $\rho_\alpha(S)$ of $S$ is the maximal root of the following quadratic polynomial in $\lambda$:
\begin{equation}\label{e7} (\lambda-\alpha)(\lambda-\alpha(m-n+k+1))-(1-\alpha)^2(m-n+1).
\end{equation}
\end{Proposition}

\begin{proof} Note that there are $m-n+1$ undirected edges in $S$. For convenience, assume that $1$ has the maximum degree $n-1$, $N^+_S(1)=[m-n+k+2]-\{1\}$ and $N^-_S(1)=([m-n+2]-\{1\})\cup \{m-n+k+3, m-n+k+4, \ldots, n\}$.

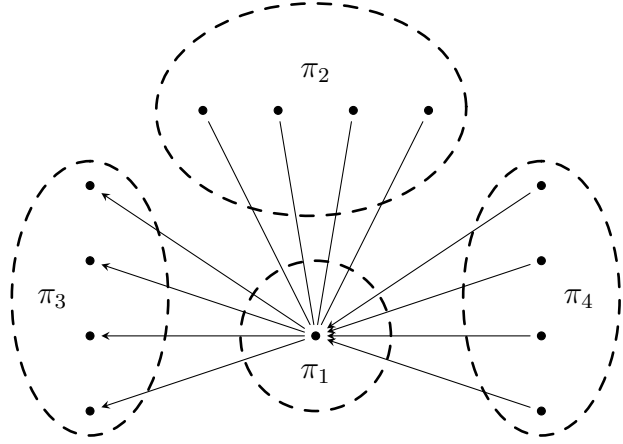
\begin{figure}
\centering
\begin{tikzpicture}[line cap=round,line join=round,>=triangle 45,x=1.0cm,y=1.0cm]
\draw  [stealth-](-0.85,9.9)-- (1.85,8.1);
\draw  [stealth-](-0.85,8.95)-- (1.85,8.05);
\draw  [stealth-](-0.85,8.)-- (1.85,8.);
\draw  [stealth-](-0.85,7.05)-- (1.85,7.95);
\draw  (0.575,10.85)-- (1.925,8.15);
\draw  (1.525,10.85)-- (1.975,8.15);
\draw  (3.425,10.85)-- (2.075,8.15);
\draw  (2.475,10.85)-- (2.025,8.15);
\draw  [stealth-](2.15,8.1)-- (4.85,9.9);
\draw  [stealth-](2.15,8.05)-- (4.85,8.95);
\draw  [stealth-](2.15,8.)-- (4.85,8.);
\draw  [stealth-](2.15,7.95)-- (4.85,7.05);
\draw [rotate around={90.:(-1.,8.5)},line width=1.pt,dash pattern=on 5pt off 5pt] (-1.,8.5) ellipse (1.825140769936451cm and 1.0397782600555754cm);
\draw [rotate around={1.145762838175079:(1.94,11.01)},line width=1.pt,dash pattern=on 5pt off 5pt] (1.94,11.01) ellipse (2.0607414666183628cm and 1.412712069828955cm);
\draw [rotate around={90.:(5.,8.5)},line width=1.pt,dash pattern=on 5pt off 5pt] (5.,8.5) ellipse (1.825140769936462cm and 1.0397782600555816cm);
\draw [line width=1.pt,dash pattern=on 5pt off 5pt] (2.,8.) circle (1.cm);
\begin{scriptsize}
\draw [fill=black] (-1.,10.) circle (1.5pt);
\draw [fill=black] (2.,8.) circle (1.5pt);
\draw [fill=black] (-1.,9.) circle (1.5pt);
\draw [fill=black] (-1.,8.) circle (1.5pt);
\draw [fill=black] (-1.,7.) circle (1.5pt);
\draw [fill=black] (0.5,11.) circle (1.5pt);
\draw [fill=black] (1.5,11.) circle (1.5pt);
\draw [fill=black] (3.5,11.) circle (1.5pt);
\draw [fill=black] (2.5,11.) circle (1.5pt);
\draw [fill=black] (5.,10.) circle (1.5pt);
\draw [fill=black] (5.,9.) circle (1.5pt);
\draw [fill=black] (5.,8.) circle (1.5pt);
\draw [fill=black] (5.,7.) circle (1.5pt);
\end{scriptsize}

\draw (2,7.5) node{$\pi_1$};
\draw (2,11.5) node{$\pi_2$};
\draw (5.5,8.5) node{$\pi_4$};
\draw (-1.5,8.5) node{$\pi_3$};
\end{tikzpicture}
\caption{The partition $\Pi$ of the vertices of a mixed star.}\label{fig2}
\end{figure}
 Set $\pi_1=\{1\},$ $\pi_2=\{2, 3, \ldots, m-n+2\},$ $\pi_3=\{m-n+3, m-n+4, \ldots, m-n+k+2\},$ and $\pi_4=[n]-\pi_1-\pi_2-\pi_3$ as illustrated in \ref{fig2}.
With respect to the partition $\Pi=\{\pi_1, \pi_2, \pi_3, \pi_4\}$ of $[m]$,  the adjacency matrix  $A$ and the diagonal out-degree matrix $D^+$ of $T$ have equitable quotient matrices $$\Pi(A)=\begin{pmatrix}
0 & m-n+1 & k & 0\\
1 & 0 & 0 & 0\\
0 & 0 & 0 & 0\\
1 & 0 & 0 & 0
\end{pmatrix}\mbox{ and }\Pi(D^+)=\begin{pmatrix}
m-n+k+1 & 0 & 0 & 0\\
0 & 1 & 0 & 0\\
0 & 0 & 0 & 0\\
0 & 0 & 0 & 1
\end{pmatrix},$$
respectively, which implies that the $A_\alpha$ matrix of $T$ has equitable quotient
$$\Pi(A_\alpha)=\begin{pmatrix}
\alpha(m-n+k+1) & (1-\alpha)(m-n+1) & (1-\alpha)k & 0\\
1-\alpha & \alpha & 0 & 0\\
0 & 0 & 0 & 0\\
1-\alpha & 0 & 0 & \alpha
\end{pmatrix}.$$
Since the characteristic polynomial of $\Pi(A_\alpha)$ is
$$\lambda(\lambda-\alpha)((\lambda-\alpha)(\lambda-\alpha(m-n+k+1))-(1-\alpha)^2(m-n+1)),$$
and the zero in (\ref{e7}) is at least $\alpha,$ we complete the proof.
\end{proof}
\bigskip

\noindent {\bf Proof of Theorem~\ref{upr}.} By Proposition~\ref{p2.5}, it suffices to show that for each maximal element $[T]\in \mathcal{T}(n, m)$ characterized in Proposition~\ref{p3.2}, $\rho_\alpha(T)$ is at most the upper bound appearing in Theorem~\ref{upr}. Suppose  $T=S$ is a mixed star with maximal out-degree $m-n+k+1$. Since the largest root of the quadratic polynomial in (\ref{e7}) increases as lone as $k$ increases, we might assume $k=2n-m-2$, and find (\ref{e7}) becomes
$$\lambda^2-\alpha n \lambda+\alpha^2(n-1)-(1-\alpha)^2(m-n+1),$$ which has largest root as the upper bound appearing in Theorem~\ref{upr}.
For the remaining mixed trees $T\in\mathcal{T}(n,n-1)$, from Corollary~\ref{n-1} we know that the $A_\alpha$ matrix of $T$ has characteristic polynomial $\prod_{i\in[n]}(\lambda-\alpha d_i^+)$, so $\rho_\alpha(T)=\alpha\cdot(\max_{i\in[n]} d_i^+)\leq\alpha (n-1)$, where the equality holds when $T$ is the star with $n-1$ leaves being out-neighbor of a vertex. Moreover, $\alpha (n-1)$ is equal to the upper bound appearing in Theorem~\ref{upr} when $m=n-1$. \qed

\section{The lower bound of $\rho_\alpha(T)$}\label{s4}

The following theorem was proved in \cite{nprs:17}.

\begin{Theorem}\label{tnprs} (\cite{nprs:17}) If $T$ is a tree of order $n$ and $\alpha\in [0, 1]$, then
$$\rho_\alpha(T)\geq \rho_\alpha(P_n).$$
Equality holds if and only if $G=P_n$. \qed
\end{Theorem}

\noindent {\bf Proof of Theorem~\ref{lwr}.} Let $T$ be a mixed tree of order $n$ and size $m$. Then $T$ has $2n-m-1$ components, and there exists a components
of order at least $k=\lceil \frac{n}{2n-m-1}\rceil$. If $m=2n-2$ then $k=n$ and $\rho_\alpha(T)\geq \rho_\alpha(P_k)$ by Theorem~\ref{tnprs} where the equality holds when $T=P_n$. For $m<2n-2$, let $C_1$ be a component of $T$ with maximum size $t$. Then $t\geq k\geq 2$ and $A_\alpha(T)[C_1]\geq A_\alpha(C_1)$.
Hence by Lemma~\ref{perron}, Lemma~\ref{char} and Theorem~\ref{tnprs},
$$\rho_\alpha(T)\geq \rho(A_\alpha(T)[C_1])\geq\rho(A_\alpha(C_1))= \rho_\alpha(P_t)\geq \rho_\alpha(P_k).$$

\qed
\bigskip

{\bf Acknowledgements.} This research is supported by the Ministry of Science and Technology
of Taiwan under the project MOST 109-2115-M-009 -007 -MY2.

\end{document}